\documentclass[11pt,makeidx]{amsart}
\usepackage{amsfonts,amssymb,amsmath,amscd}
\usepackage{tikz}
\usetikzlibrary{positioning}
\usetikzlibrary{matrix,arrows,decorations.pathmorphing}
\font\cyr=wncyr10

\newcommand{\A}{{\mathbb A}}
\newcommand{\C}{{\mathbb C}}
\newcommand{\F}{{\mathbb F}}
\newcommand{\Q}{{\mathbb Q}}

\newcommand{\Qp}{{\Q_p}}
\newcommand{\bQ}{{\overline\Q}}

\newcommand{\R}{{\mathbb R}}
\newcommand{\T}{{\mathbb T}}

\newcommand{\Z}{{\mathbb Z}}
\newcommand{\Zp}{{\Z_p}}

\newcommand{\grl}{{\mathfrak l}}
\newcommand{\grn}{{\mathfrak n}}
\newcommand{\grp}{{\mathfrak p}}

\newcommand{\eps}{{\epsilon}}
\newcommand{\veps}{{\varepsilon}}

\newcommand{\CK}{{\mathcal K}}
\newcommand{\CL}{{\mathcal L}}
\renewcommand{\O}{{\mathcal O}}

\newcommand{\alg}{\mathrm{alg}}
\newcommand{\Aut}{\mathrm{Aut}}
\newcommand{\Div}{\mathrm{Div}}
\newcommand{\End}{\mathrm{End}}

\newcommand{\Gal}{\mathrm{Gal}}
\newcommand{\GL}{\mathrm{GL}}
\newcommand{\Hom}{\mathrm{Hom}}
\newcommand{\im}{\mathrm{im}}
\newcommand{\Ind}{\mathrm{Ind}}
\newcommand{\isoarrow}{\stackrel{\sim}{\rightarrow}}

\newcommand{\Nm}{\mathrm{Nm}}
\newcommand{\ord}{\mathrm{ord}}
\newcommand{\rank}{\mathrm{rank}}
\newcommand{\res}{\mathrm{res}}
\newcommand{\Res}{\mathrm{Res}}
\newcommand{\rec}{\mathrm{rec}}
\newcommand{\Sel}{\mathrm{Sel}}
\newcommand{\Sha}{\hbox{\cyr X}}

\makeatletter
\@addtoreset{equation}{subsection}

\makeatother

\newtheorem{coro}[subsubsection]{\bf Corollary}
\newtheorem{lem}[subsubsection]{\bf Lemma}

\newtheorem{prop}[subsubsection]{\bf Proposition}
\newtheorem{conj}[subsubsection]{\bf Conjecture}

\newtheorem{innercustomthm}{Theorem}
\newenvironment{customthm}[1]
  {\renewcommand\theinnercustomthm{#1}\innercustomthm}
  {\endinnercustomthm}

\theoremstyle{definition}
\newtheorem{rmk}[subsubsection]{\it Remark}

\author{Christopher Skinner}

\address{
 Department of Mathematics\\
 Princeton University\\
  Fine Hall, Washington Road \\
 Princeton, NJ 08544-1000\\
 USA}

\title[A converse to a theorem of Gross, Zagier, and Kolyvagin]{A converse to a theorem of Gross, Zagier, \\ and Kolyvagin}

\begin{document}
\setcounter{section}{0}
\begin{abstract}
Let $E$ be a semistable elliptic curve over $\Q$. We prove that if $E$ has non-split multiplicative reduction at at least one odd prime or split multiplicative reduction at at least two odd primes, then
$$
\rank_\Z E(\Q) =1 \ \text{and} \ \#\Sha(E)<\infty \implies \ord_{s=1}L(E,s)=1.
$$
We also prove the corresponding result for the abelian variety associated with a weight two newform $f$ of 
trivial character. These, and other related results, are consequences of our main theorem, which establishes
criteria for $f$ and $H^1_f(\Q,V)$, where $V$ is the $p$-adic Galois representation associated with $f$, 
that ensure that $\ord_{s=1}L(f,s)=1$. The main theorem is proved using the Iwasawa theory of $V$ 
over an imaginary quadratic field to show that the $p$-adic logarithm of a suitable Heegner point is non-zero.
\end{abstract}
\maketitle

\section{Introduction}

Let $f\in S_2(\Gamma_0(N))$ be a newform with trivial nebentypus.
Associated with $f$ is an abelian variety $A_f$ over $\Q$ (really an
isogeny class of abelian varieties) characterized by an equality
of the Hasse-Weil $L$-function $L(A_f,s)$ of $A_f$ and the product
of the $L$-functions $L(f^\sigma,s)$ of all the Galois conjugates
$f^\sigma$ of $f$:
$$
L(A_f,s)=\prod L(f^\sigma,s).
$$
The endomorphism ring $\End^0_\Q (A_f)$ is a totally real field
$M_f$ of degree equal to the dimension of $A_f$; this is naturally identified
with the subfield of $\C$
generated over $\Q$ by the Hecke eigenvalues of $f$ (equivalently, the Fourier
coefficients of the $q$-expansion of $f$ at $\infty$), and so its degree
$[M_f:\Q]$ is equal to the number of conjugate forms $f^\sigma$.

The celebrated Birch--Swinnerton-Dyer conjecture, as formulated for
abelian varieties, 
predicts that the rank of $A_f(\Q)$ equals the order of vanishing at $s=1$ of the $L$-function $L(A_f,s)$:
\begin{equation*}
\rank_\Z A_f(\Q) \stackrel{?}{=} \ord_{s=1} L(A_f,s).
\end{equation*}
The most spectacular result to date in the direction of this conjecture follows
from the combination of the work of Gross and Zagier \cite{GZ} and Kolyvagin \cite{Koly}
\cite{Koly-Log},
which together prove:
Let $r=0$ or $1$. Then
$\ord_{s=1}L(f,s)=r$ if and only if $\ord_{s=1}L(A_f,s)= [M_f:\Q]r$, and
\begin{equation*}
 \ord_{s=1} L(f,s)=r \implies \begin{matrix} \rank_\Z A_f(\Q)=[M_f:\Q]r \ \text{and} \
\#\Sha(A_f)< \infty,
\end{matrix}
\end{equation*}
where $\Sha(A_f)$ is the Tate--Shafarevich group of $A_f$ (conjecturally always finite).
For $r=0$ the converse to this implication was established in \cite{SU-MCGL} (for $N$ not squarefull) and
in \cite{Xin-thesis} (for all $N$). In the proofs of these converses it is only needed that the $p$-primary part
$\Sha(A_f)[p^\infty]$ be finite for a sufficiently large prime $p$ such that $f$ is
ordinary with respect to some prime $\lambda\mid p$ of $M_f$.
In particular, if $A_f$ is a semistable elliptic curve, then it suffices to take $p\geq 11$
of good, ordinary reduction.

For $r=1$ and $A_f$ an elliptic curve having complex multiplication, a converse to the above implication
is a consequence of the combination of results of Rubin, Bertrand, and Perrin-Riou; this is explained
in Theorems 8.1 and 8.2 of \cite{Ru-CMrational}.
In this paper we prove a converse for the $r=1$ case for those $A_f$ with $N$ squarefree. If $N$ is squarefree, then
$A_f$ does not have complex multiplication, so the cases covered by the theorems in this paper are disjoint from those
covered by the results recalled in \cite{Ru-CMrational}.

Let $\pi=\otimes\pi_v$ be the cuspidal automorphic representation of $\GL_2(\A)$ such that $L(\pi,s-1/2)=L(f,s)$ 
(so $N$ is the conductor of $\pi$).

\begin{customthm}{A}\label{thmA} Suppose $N$ is squarefree.
If there is at least one odd prime $\ell$ such that $\pi_\ell$ is the twist of the special
representation by the unique unramified quadratic character or at least two odd primes
$\ell_1$ and $\ell_2$ such that $\pi_{\ell_1}$ and $\pi_{\ell_2}$ are special,
then
\begin{equation*}
\begin{matrix} \rank_\Z A_f(\Q)=[M_f:\Q] \ \text{and} \ \#\Sha(A_f)< \infty \implies \ord_{s=1} L(f,s) = 1.
\end{matrix}
\end{equation*}
\end{customthm}

\noindent The hypotheses on the local representations can also be formulated in terms of the
eigenvalues of the Atkin-Lehner involutions $w_\ell$ acting on $f$: either there is at least one odd prime $\ell$
such that the eigenvalue of $w_\ell$ is $+1$ or there are at least two odd primes $\ell_1$ and $\ell_2$
such that the eigenvalues of $w_{\ell_1}$ and $w_{\ell_2}$ are both $-1$.

Since every elliptic curve over $\Q$ is modular, Theorem A can be restated in this case to read:
\begin{customthm}{A$'$}\label{thmA'}
Suppose $E$ is a semistable elliptic curve over $\Q$. If
there is at least one odd prime at which $E$ has nonsplit multiplicative reduction or at least
two odd primes at which $E$ has split multiplicative reduction,
then
\begin{equation*}
\begin{matrix} \rank_\Z E(\Q)=1 \ \text{and} \ \#\Sha(E)< \infty \implies \ord_{s=1} L(E,s) = 1.
\end{matrix}
\end{equation*}
\end{customthm}

As the hypotheses on $E(\Q)$ and $\Sha(E)$ ensure that
the root number of $E$ is $-1$ (by, for instance, Nekov\'a\v{r}'s results toward the parity conjecture),
the hypotheses on the reduction types in Theorem A$'$
only exclude those semistable curves that have conductor equal to $2\ell$ with $\ell$ an odd prime and
that have
split reduction at both $2$ and $\ell$. However, showing that a
positive proportion of semistable elliptic curves $E$, when ordered by height for example, 
satisfy the hypotheses of Theorem \ref{thmA'} with $\rank_\Z E(\Q)=1$ and $\#\Sha(E)<\infty$ remains an open
and interesting problem.

Our proof of Theorem A is similar in spirit to those of Theorems 8.1 and 8.2 of \cite{Ru-CMrational}
insofar as it is essentially $p$-adic.
We deduce Theorem A from Theorem B below, which gives a $p$-adic criterion for
$A_f$ to have
both algebraic and analytic rank $[M_f:\Q]$ over an imaginary quadratic field.
However, unlike the proofs in \cite{Ru-CMrational}, our proof of this criterion does 
not make use of a $p$-adic Gross--Zagier formula for the derivative of
a $p$-adic $L$-function or require non-degeneracy
of $p$-adic heights. Instead it uses a formula expressing the value of a $p$-adic
$L$-function in terms of the formal log of a rational point on $A_f$. 

Let $p$ be a prime and $\lambda\mid p$ a prime of $M_f$. Let $L= M_{f,\lambda}$. 
Let $\rho_{f,\lambda}:\Gal(\bQ/\Q)\rightarrow \Aut_L(V)$ be the usual two-dimensional
Galois representation associated with $f$ and let $\bar\rho_{f,\lambda}$ be
its residual representation (the semisimplification of the reduction of a Galois
stable lattice in $V$).

\begin{customthm}{B}\label{thmB}
Suppose $N$ is squarefree and $p\geq 5$.
Let $K$ be an imaginary quadratic field with odd discriminant $D$.
Suppose \begin{itemize}
\item[(a)] $p\nmid N$ and $f$ is ordinary with respect to $\lambda$;
\item[(b)] $\bar\rho_{f,\lambda}$ is irreducible and
ramified at an odd prime that is either inert or ramified in $K$;
\item[(c)] both $2$ and $p$ split in $K$;
\item[(d)] if $(D,N)\neq 1$ then for each $\ell\mid (D,N)$, 
$\pi_\ell$ is the twist of the special representation by the unique unramified quadratic character,
and each prime divisor of  $N/(D,N)$ splits in $K$;
\item[(e)] $\dim_L H^1_f(K,V)=1$ and the restriction
$H^1_f(K,V)\stackrel{\res}{\rightarrow}\prod_{v\mid p}H^1_f(K_v,V)$ is an injection.
\end{itemize} Then $\ord_{s=1}L(f,K,s)=1$, $\rank_\Z A_f(K)=[M_f:\Q]=\ord_{s=1}L(A_f/K,s)$, and $\Sha(A_f/K)$ is finite.
\end{customthm}

\noindent Here
$H^1_f(K,V)$ and $H^1_f(K_v,V)$ are, respectively, the global and local Bloch--Kato
Selmer groups. Also, $L(A_f/K,s)$ and $\Sha(A_f/K)$ are the $L$-function and the
Tate--Shafarevich group of $A_f/K$, and $L(f,K,s)=L(f,s)L(f\otimes\chi_D,s)$ with
$\chi_D$ the quadratic character associated with $K$.

The deduction of Theorem A from Theorem B, which is explained in more detail in Section \ref{AimpliesB}, goes as follows: 
As $f$ is not a CM form (since $N$ is squarefree), condition (a) holds for
some $\lambda\mid p$ for a set of primes $p$ of density one, while $\bar\rho_{f,\lambda}$
is irreducible if
$p$ is sufficiently large,
in which case it follows from Ribet's work on level-lowering that
$\bar\rho_{f,\lambda}$ is ramified at some odd prime $q\neq p$ and even, for $p$ sufficiently large,
ramified at all primes $q\mid \mid N$. Fix such a $p$ and $\lambda$.
From the hypotheses on $\pi_\ell$ in Theorem A it
then follows that an imaginary quadratic field $K$ can be chosen so that
its discriminant $D$ is odd; (b),
(c), and (d) hold;
and $L(A_f^D,1)\neq 0$,
where $A_f^D$ is the $K$-twist of $A_f$.
As $L(A_f^D,s) = \prod L(f^\sigma\otimes\chi_D,s)$,
the existence of a $K$ with the desired properties
follows from \cite{Friedberg-Hoffstein}.
From the non-vanishing of $L(A_f^D,1)$ it follows that
$A_f^D(\Q)$ and $\Sha(A_f^D)$ are finite. Together with $A_f(\Q)$ having rank $[M_f:\Q]$ and $\Sha(A_f)$ being
finite, this implies that $A_f(K)$ has rank $[M_f:\Q]$ and $\Sha(A_f/K)$ is finite, which in turn imply (e).
It then follows from Theorem B that $\ord_{s=1}L(A_f/K,s) = [M_f:\Q]$
and $\ord_{s=1}L(f,K,s)=1$. As $L(A_f/K,s) = L(A_f,s)L(A_f^D,s)$
and $L(f,K,s)=L(f,s)L(f\otimes\chi_D,s)$, it follows that $\ord_{s=1}L(A_f,s) = [M_f:\Q]$
and $\ord_{s=1}L(f,s)=1$.

As the deduction of Theorem A from Theorem B shows, the hypothesis that
$\Sha(A_f)$ is finite can be replaced
with $\Sha(A_f)[p^\infty]$ finite for some suitable prime $p$. It is even possible to formulate
conditions on the $\lambda$-Selmer group of $A_f/K$ that ensure that hypothesis (e) of 
Theorem \ref{thmB} holds, from which one can deduce, for example:

\begin{customthm}{C}\label{thmC} Let $E$ be a semistable curve over $\Q$ such that 
there is at least one odd prime at which $E$ has nonsplit multiplicative reduction or at least
two odd primes at which $E$ has split multiplicative reduction. Suppose there is 
a prime $p\geq 5$ at which $E$ has 
good, ordinary reduction and such that 
\begin{itemize}
\item[(a)] $E[p]$ is an irreducible $\Gal(\bQ/\Q)$-representation;
\item[(b)] $\Sel_p(E)\cong \Z/p\Z$;
\item[(c)] the image of the restriction map $\Sel_p(E)\rightarrow E(\Qp)/pE(\Qp)$ does not
lie in the image of $E(\Qp)[p]$.
\end{itemize}
Then $\ord_{s=1}L(E,s)=1= \rank_\Z E(\Q)$ and $\#\Sha(E)<\infty$.
\end{customthm}
\noindent It is through similar variations that Theorem \ref{thmB} plays a crucial role in a recent proof that:

\begin{customthm}{D}[Bhargava-Skinner \cite{Bh-Sk-posprop}] A positive proportion of
elliptic curves, when ordered by height, have both algebraic and analytic rank one.
\end{customthm}

To explain how to pass from the hypotheses of Theorem B to its conclusion, we begin by recalling the
Gross--Zagier formula.  It follows from hypothesis (e) and the parity
conjecture for Selmer groups of $p$-ordinary modular forms \cite{Ne-parity}
that the sign of the functional equation of $L(f,K,s)$
is $-1$.  It then follows from the general Gross--Zagier formula of X.~Yuan, S.~Zhang, and
W.~Zhang \cite{YZZ} that $A_f$ is a quotient of the Jacobian of some Shimura curve
(possibly a modular curve)
such that the N\'eron-Tate height of the image $P_K(f)\in A_f(K)\otimes M_f$ of
a certain $0$-cycle on the curve (essentially a Heegner cycle) is related to $L'(f,K,1)$ by
\begin{equation*}
\langle P_K(f),P_K(f)\rangle \stackrel{.}{=} L'(f,K,1),
\end{equation*}
where $\langle \cdot,\cdot\rangle$ is the N\'eron-Tate height-pairing (relative to some
symmetric ample line bundle) and `$\stackrel{.}{=}$' denotes equality up to 
a non-zero constant (which depends on $f$ and the line bundle). 
So if we expect to prove that
$\ord_{s=1}L(f,K,1)=1$ then we should expect to prove that the height of 
$P_K(f)$ is non-zero or, equivalently, that $P_K(f)\neq 0$.
If $P_K(f)\neq 0$, then (e) together with the general Gross--Zagier formula
and the work of Kolyvagin implies the conclusions of Theorem B. We establish Theorem B by
proving that $P_K(f)\neq 0$.

To show that $P_K(f)\neq 0$ we do not directly show that its height is non-zero.
Instead we show that its formal logarithm at a prime of $K$ above $p$ does not
vanish, which is sufficient for our purposes.
To do this we make use of $p$-adic analogs of the Gross--Zagier formula, proved by
Bertolini, Darmon, and Prasanna and Brooks, which are analogs of a formula proved
by Rubin \cite{Ru-CMrational} in the CM case. Recall that $p=\grp\bar \grp$ splits in $K$ and
that $D$ is odd. As explained in
\cite{BDP} and \cite{Hunter} there is a $p$-adic $L$-function $L^S_\grp(f,\chi)$,
a function of certain $p$-adic anti-cyclotomic
Hecke characters $\chi$ of $K$, such that
\begin{equation*}
L^S_\grp(f,1) \stackrel{.}{=} (\log_\omega P_K(f))^2,
\end{equation*}
where $\log_\omega:A_f(K_\grp)\otimes L\rightarrow \overline K_\grp$ is the formal logarithm,
determined
by a certain 1-form $\omega\in \Omega^1(A_f)\otimes M_f$, and 
`$\stackrel{.}{=}$' again denotes equality up to 
a non-zero constant.
Our aim then is to show that $L_\grp^S(f,1)\neq 0$ under the
hypotheses of Theorem B. Our method for doing so is via Iwasawa theory.

Iwasawa theory conjecturally relates the $p$-adic $L$-function $L^S_\grp(f,\chi)$ to the characteristic
ideal of a certain $p$-adic Selmer group. One consequence of such a relation would be
the implication
\begin{equation*}
L^S_\grp(f,1) =0 \implies H^1_\grp(K,V)\neq 0,
\end{equation*}
where $H^1_\grp(K,V) \subset H^1(K,V)$ is the subspace of classes that vanish in
$H^1(K_w,V)$ for all places $w\neq \bar\grp$.  However, hypothesis (e) of
Theorem B ensures
that $H^1_\grp(K,V)=0$, so it would follow from this implication that $L^S_\grp(f,1)\neq 0$
and hence that $P_K(f)$ is non-torsion.  Our strategy for proving Theorem B ultimately reduces to
the above implication. The desired result from Iwasawa theory is part of recent work of
Wan \cite{Wan-U31}, following the methods of
\cite{SU-MCGL}, under certain hypotheses on $f$, $p$, and $K$.
The conditions (a)-(d) of Theorem B ensure that these hypotheses hold.

Following the proof of Theorem B, we include remarks emphasizing where in the arguments the
various hypotheses intervene, with an eye toward future developments that should remove many
of them. We then elaborate on the deduction of Theorem \ref{thmA} from Theorem \ref{thmB}
and explain how similar arguments can be applied to the $r=0$ case, giving an alternate
proof of a special case of the results in \cite{SU-MCGL} and \cite{Xin-thesis}.

After the first version of this paper was completed, Wei Zhang released
a preprint \cite{Z-Koly} in which he proves many cases of a conjecture of Kolyvagin, showing that the $p$-adic 
Selmer group of $A_f$ is often spanned by classes derived from Heegner points. As a consequence Zhang 
obtains a theorem similar to Theorem \ref{thmB}. This theorem does not require the restriction map at primes
above $p$ be injective (the second half of hypothesis (e) of Theorem \ref{thmB}) but crucially requires that the Tamagawa factors
at the primes that split in $K$ or are congruent to $\pm 1$ modulo $p$ be indivisible by $p$. Theorem \ref{thmB}
imposes no hypotheses on Tamagawa factors. While there is substantial overlap in the cases covered by Theorem \ref{thmB}
and the results of \cite{Z-Koly}, neither subsumes the other. The proof of the main result in \cite{Z-Koly} 
also relies on Iwasawa theory, in this case on consequences of the Main Conjecture for $\GL_2$ proved in \cite{SU-MCGL}.

\noindent{\em Acknowledgements.}
The author thanks Kartik Prasanna and Xin Wan for helpful conversations.
This work was partially supported by grants
from the National Science Foundation, including DMS-0701231 and DMS-0758379.
The first version of this paper was written while the author was a visitor in the School of Mathematics
at the Institute for Advanced Study.

\section{The proof of Theorem \ref{thmB}}

Let $\bQ$ be an algebraic closure of $\Q$ and $K/\Q$ an imaginary quadratic field in $\bQ$.
Fix an embedding
$\bQ\hookrightarrow\C$. This determines a complex conjugation $c\in G_\Q=\Gal(\bQ/\Q)$, which induces the non-trivial 
automorphism on $K$.
For each prime $\grl$ of $K$ let $\overline K_\grl$ be an algebraic closure of $K_\grl$ and fix an
embedding $\bQ\hookrightarrow \overline K_\grl$; the latter
realizes $G_{K_\grl}=\Gal(\overline K_\grl/K_\grl)$ as a decomposition subgroup for $\grl$ in
$G_K=\Gal(\bQ/K)$. Let $I_\grl\subset G_{K_\grl}$ be the inertia subgroup.
Let $\F_\grl$ be the residue field of $K_\grl$ and $\overline\F_\grl$ the residue field of $\overline K_\grl$ (so $\overline\F_\grl$ is an 
algebraic closure of $\F_\grl$); there is then a canonical isomorphism $G_{K_\grl}/I_\grl\isoarrow G_{\F_\grl} = \Gal(\overline\F_\grl/\F_\grl)$.

Let $p$ be an odd prime.

\subsection{Modular forms and abelian varieties}
Let $f\in S_2(\Gamma_0(N))$ be a newform with trivial Nebentypus.
Let $M_f$ be the subfield of $\C$ generated by the Hecke eigenvalues of $f$ (equivalently,
the Fourier coefficients of the $q$-expansion of $f$ at the cusp $\infty$); this
is a totally real number field, and the fixed embedding $\bQ\hookrightarrow\C$ identifies
$M_f$ with a subfield of $\bQ$.

A construction of Eichler and Shimura associates with $f$ an abelian variety $A_f$ over $\Q$
of dimension $[M_f:\Q]$ and such that $\End^0_\Q(A_f)$ is naturally identified with $M_f$ and
characterized (up to isogeny) by
$$
L(A_f,s) = \prod_{\sigma:M_f\hookrightarrow \C}L(f^\sigma,s),
$$
where $f^\sigma$ is the conjugate of $f$; that is, the newform in $S_2(\Gamma_0(N))$ whose
$q$-expansion at $\infty$ has coefficients
obtained by applying $\sigma$ to those of $f$.

Let $T_pA_f$ be the $p$-adic Tate-module of $A_f$ and let $V_pA_f = T_pA_f\otimes_\Zp\Qp$.
Let $\lambda$ be a prime of $M_f$ above $p$ and let $M_{f,p}=M_f\otimes\Qp$
and let $L$ be a finite (field) extension of $M_{f,\lambda}$. Then
$$
V=V_pA_f\otimes_{M_{f,p}}L
$$
is a two-dimensional $L$-space with a continuous, $L$-linear $G_\Q$-action, which
we denote by $\rho_{f,\lambda}$.
It is potentially semistable at $p$, unramified at $\ell\nmid Np$, and such that
$V^\vee\cong V(-1)$ (the $-1$-Tate twist of $V$). Furthermore, if we fix an
embedding $L\hookrightarrow \C$ that agrees with the inclusion $M_f\hookrightarrow \C$,
then\footnote{Our conventions for $L$-functions of potentially semistable Galois representations of $G_\Q$ or $G_K$ are geometric: the local Euler factors are defined using the characteristic polynomials of geometric Frobenius elements.} $L(V^\vee,s)=L(f,s)$.

Recall that $f$ is ordinary with respect to $\lambda$ if the eigenvalue $a_p(f)$ of the action on $f$
of
the Hecke operator $T_p$, or $U_p$ if $p\mid N$, (equivalently, the $p$th Fourier coefficient of the $q$-expansion
at $\infty$) is a unit at $\lambda$; that is, if $a_p(f)$ is a unit in the ring of integers of $L$.

By the $K$-twist of $A_f$ we mean the abelian variety $A_f^D$ over $\Q$ obtained by twisting
by the cocycle in $H^1(\Q,\Aut_\Q A_f)$ defined by the quadratic character $\chi_D:G_\Q\rightarrow
\{\pm 1\}\subset\Aut_\Q A_f$ associated with $K$ (so $G_K$ is the kernel of $\chi_D$). Then
$$
V_D = V_pA_f^D\otimes_{M_{f,p}}L \cong V\otimes\chi_D
$$
as continuous $L$-linear representations of $G_\Q$ and\footnote{For convenience
we will identify the Galois character $\chi_D$ with the quadratic Dirichlet character
of the same conductor.}
$$
L(A_f^D,s) = \prod_{\sigma:M_f\hookrightarrow\C} L(f^\sigma\otimes\chi_D,s).$$ 
The
natural map $A_f\times A_f^D\rightarrow \Res_{K/\Q} A_f$ is a $\Q$-isogeny with kernel and cokernel
annihilated by $2$.

Let $\eps(f)\in \{\pm 1\}$ be the sign of the functional equation of $L(f,s)$.
Let
$$
L(f,K,s) = L(f,s)L(f\otimes\chi_D,s).
$$
The sign of the functional equation of $L(f,K,s)$ is then $\eps(f,K)=\eps(f)\eps(f\otimes\chi_D)$.

Let $\pi=\otimes\pi_v$ be the cuspidal automorphic representation of $\GL_2(\A)$ such that
$L(\pi,s-1/2) = L(f,s)$.
Then $L(f,K,s) = L(BC_K(\pi),s-1/2)$, where $BC_K(\pi)$ is the base change
of $\pi$ to an automorphic representation of $\GL_2(\A_K)$. For a prime $\grl\mid\ell$ 
of $K$ over $\ell$, we also write $BC_{K_\grl}(\pi_\ell)$ for the base change of 
$\pi_\ell$ to an admissible representation of $\GL_2(K_\grl)$; so 
$BC_{K_\grl}(\pi_\ell)$ is the $\grl$-constituent of $BC_K(\pi)$.

Let $\eps(\pi,K)=\eps(BC_K(\pi),1/2)$
be the global root number of $BC_K(\pi)$. Similarly, for a prime $\ell$,
let $\eps_\ell(\pi,K) = \prod_{\grl\mid \ell}\eps(BC_{K_\grl}(\pi_\ell),1/2)$ be the product
of the local root numbers. Then
$$
\eps(f,K)= \eps(\pi,K) = - \prod_\grl \eps(BC_{K_\grl}(\pi_\ell),1/2) = - \prod_\ell \eps_\ell(\pi,K).
$$
If $\ell$ splits in $K$, then $\eps_\ell(\pi,K)=\eps(\pi_\ell,1/2)^2=+1$ since $BC_{K_\grl}(\pi_\ell)\cong \pi_\ell$, 
so the local contribution to the global
sign $\eps(f,K)$ comes only from primes that are inert or ramified in $K$. Furthermore, if
$\pi_\ell$ is the special representation $\sigma_\ell$, then $\eps_\ell(\pi,K)=-1$ if
$\ell$ is inert or ramified as there is then only one prime $\grl$ of $K$ above $\ell$
and $BC_{K_\grl}(\sigma_\ell)$ is the special representation. 
And if $\pi_\ell$ is the twist $\sigma_\ell\otimes\xi_\ell$
of the special representation by the unique unramified quadratic character $\xi_\ell$,
then $\eps_\ell(\pi,K)=-1$ if $\ell$ is inert, as $BC_{K_\grl}(\pi_\ell)$ is then 
the special representation, and $\eps_\ell(\pi,K)=+1$ if $\ell$ is ramified, as
$BC_{K_\grl}(\pi_\ell)$ is then the twist of the special representation by the unique
unramified quadratic character of $K_\grl$. Here we have used that the root number of
the special representation is $-1$ and the root number of the twist of the special
representation by the unramified quadratic extension is $+1$;
\cite{KramerTunnel} and \cite[Props.~3.5 and 3.6]{JL} are useful references
for these and other facts about epsilon factors and root numbers.

\subsection{Selmer groups}\label{Selmer} Bloch and Kato \cite{BK} (see also \cite{FP}) defined
Selmer groups for geometric $p$-adic Galois representations. For the representation $V$ this
Selmer group is
$$
H^1_f(K,V) = \ker\{H^1(K,V)\stackrel{\res}{\rightarrow}\prod_\grl H^1(K_\grl,V)/H^1_f(K_\grl,V)\},
$$
where $H^1_f(K_\grl,V)=\ker\{H^1(K_\grl,V)\rightarrow H^1(K_\grl,B_{cris}\otimes_\Qp V)\}$,
with $B_{cris}$ the ring of crystalline periods, if $\grl\mid p$,
and $H^1_f(K_\grl,V) = H^1(\F_\grl,V^{I_\grl})$ if $\grl\nmid p$.
By Tate's local Euler characteristic formula and local duality, if $\grl\nmid p$,
\begin{equation*}
\dim_L H^1(K_\grl, V) = \dim_L H^0(K_\grl,V)+\dim_LH^2(K_\grl,V) = 2\dim_L H^0(K_\grl,V),
\end{equation*}
where we have used $V\cong V^\vee(1))$ in the second equality. If follows from 
the local-global compatibility of $V$ with $\pi_\ell$ that
$H^0(K_\grl,V) = V^{G_{K_\grl}}=0$
(cf.~\cite[Lem.~3.1.3]{N-P}), whence $H^1_f(K_\grl,V) = H^1(K_\grl,V) =0$.

Let $S$ be any finite set of primes containing those at which $V$ is ramified (so
those dividing $pN$) and $G_{K,S}$ the
Galois group over K of the maximal extension of $K$ in $\bQ$ that is unramified outside the prime ideals dividing those in $S$.
Since $H^1(K_\grl,V)=0$ if $\grl\nmid p$,
$$
H^1_f(K,V) = \ker\{H^1(G_{K,S},V)\stackrel{\res}{\rightarrow} \prod_{\grl\mid p} 
H^1(K_\grl,V)/H^1_f(K_\grl,V)\}.
$$

The same definitions can, of course, be made with $V$ replaced by $V_D$ as well as with
$K$ replaced by $\Q$ and the primes $\grl$ replaced
with rational primes $\ell$. Then the restriction map from $G_\Q=\Gal(\bQ/\Q)$ to
$G_K$ induces identifications
\begin{gather*}
H^1_f(\Q,V)\oplus H^1_f(\Q,V_D)\isoarrow H^1_f(K,V).
\end{gather*}
The Galois group $\Gal(K/\Q)$ acts on the right-hand side above,
and the left-hand side is identified with the decomposition into the sum of the subgroups
on which $\Gal(K/\Q)$ acts trivially and non-trivially, respectively.

\begin{lem}\label{signlemma} Suppose $p\nmid ND$ and $f$ is ordinary
with respect to $\lambda\mid p$.
If $\dim_L H^1_f(K,V)$ is odd then $\eps(f,K)$ is $-1$.
\end{lem}

\begin{proof} Since
$$
\dim_L H^1_f(K,V) = \dim_L H^1_f(\Q,V)+\dim_L H^1_f(\Q,V_D)
$$
is odd, one
of $\dim_L H^1_f(\Q,V)$ and $\dim_L H^1_f(\Q,V_D)$ is odd and the other is even.
It then follows from the parity conjecture for the Selmer groups of modular forms that are ordinary
at $\lambda$, proved by Nekov\'a\v{r} \cite[Thm.~12.2.3]{Ne-parity}, that one of the signs $\eps(f)$ and $\eps(f\otimes\chi_D)$
is $-1$ and the other is $+1$. Then $\eps(f,K)= \eps(f)\eps(f\otimes\chi_D)=-1$.
\end{proof}
\vskip 0.1in

\noindent{\bf Connections with the Selmer group of $A_f$.}
 Recall that the $p^\infty$-Selmer group of $A_f/K$ is
 $$
 \Sel_{p^\infty}(A_f/K) = \ker\{H^1(K,A_f[p^\infty])\stackrel{\res}{\rightarrow}\prod_{\grl}
 H^1(K_\grl,A_f(\overline K_\grl))\}
 $$
 and the $p$-primary part of the Tate--Shafarevich group of $A_f/K$ is
 $$
 \Sha(A_f/K)[p^\infty] = \ker\{H^1(K,A_f(\bQ))[p^\infty]\stackrel{\res}{\rightarrow}\prod_{\grl}
 H^1(K_\grl,A_f(\overline K_\grl))\}
 $$
 and that these sit in the fundamental exact sequence:
 \begin{equation*}
 0\rightarrow A_f(K)\otimes\Qp/\Zp\rightarrow \Sel_{p^\infty}(A_f/K) \rightarrow
 \Sha(A_f/K)[p^\infty]\rightarrow 0.
 \end{equation*}

 Let $W=V_pA_f/T_pA_f = A_f[p^\infty]$ (the last identification being $(x_n)\otimes\frac{1}{p^m}
 \mapsto x_m$). Then $\Sel_{p^\infty}(A_f/K)$ consists of those classes
 with restriction at each prime $\grl$ in the image $H^1_f(K_\grl,W)$ of
 the Kummer map $A_f(K_\grl)\otimes\Qp/\Zp\hookrightarrow H^1(K_\grl,W)$.
 Bloch and Kato proved that this subgroup is just the image of $H^1_f(K_\grl,V_pA_f)$ in $H^1(K_\grl,W)$,
where $H^1_f(K_\grl,V_pA_f)$ is defined just as $H^1_f(K_\grl,V)$
(note that $H^1(K_\grl,V_pA_f)$, $H^1_f(K_\grl,V_pA_f)$, and $H^1_f(K_\grl,W)$ are all $0$ if $\grl\nmid p$).
 In particular, if $S$ is a finite set of primes of $K$ containing all those that divide $pN$, 
 then $\Sel_{p^\infty}(A_f/K)\subset H^1(G_{K,S},W)$ consists of those
 classes with restriction to $H^1(K_\grl,W)$ belonging to $H^1_f(K_\grl,W)$ for all $\grl\in S$.
 As the image of $H^1(G_{K,S},V_pA_f)$ in $H^1(G_{K,S},W)$ is the maximal divisible subgroup
 (and has finite index), it follows that
 the maximal divisible subgroup of $\Sel_{p^\infty}(A_f/K)$ is the image in $H^1(K,W)$ of the
 characteristic zero Bloch--Kato Selmer group $H^1_f(K,V_pA_f)$.

 The projective limit of the Kummer maps for the multiplication by $p^n$-maps yields an injection
 $$
 A_f(K)\otimes\Qp\hookrightarrow H^1_f(K,V_pA_f)
 $$
 that is compatible with the fundamental exact sequence, so the cokernel
 has $\Qp$-dimension equal to the corank of $\Sha(A_f/K)[p^\infty]$.

The connection with $H^1_f(K,V)$ is just
$$
H^1_f(K_\grl,V) = H^1_f(K_\grl,V_pA_f)\otimes_{M_{f,p}}L \ \ \text{and} \ \
H^1_f(K,V)=H^1_f(K,V_pA_f)\otimes_{M_{f,p}}L,
$$
from which, together with the fact that $A_f(K)\otimes\Q_p$ is a free $M_{f,p}$-module, we deduce:

\begin{lem}\label{rank1lemma}
If $\rank_\Z A_f(K) = [M_f:\Q]$ and $\#\Sha(A_f/K)[p^\infty] <\infty$, then
$\dim_L H^1_f(K,V)=1$ and the restriction map
$H^1_f(K,V)\stackrel{\res}{\rightarrow} H^1(K_\grl,V)$ is an injection for each $\grl\mid p$.
\end{lem}

\subsection{More Galois cohomology}\label{moreGalois}
Let $S$ be any finite set of primes containing those dividing $pN$, and let
$$
H^i(K_p,V) = \prod_{\grl\mid p} H^i(K_\grl,V), \ \ \text{and} \ \
H^1_{str}(K,V) = \ker\{H^1(G_{K,S},V)\stackrel{\res}{\rightarrow} H^1(K_p,V)\}.
$$
Note that $H^1_{str}(K,V)$ (often called the `strict' Selmer group of $V$) is independent of $S$.

\begin{lem}\label{dim2lemma} $\dim_L\mathrm{im}\{H^1(G_{K,S},V)\stackrel{\res}{\rightarrow} H^1(K_p,V)\}=2$.
\end{lem}

\begin{proof} By Tate global duality, $H^1_{str}(K,V)$ is dual to $H^2(G_{K,S},V)$ (here 
we are using that $H^1(K_\grl,V)=0$ if $\grl\nmid p$ and $H^2(K_\grl,V)=0$ for all $\grl$),
so 
\begin{equation*}\begin{split}
\dim_L H^1(G_{K,S},V) - & \dim_L H^1_{str}(K,V) \\
& = \dim_{L} H^1(G_{K,S},V) - \dim_{L} H^2(G_{K,S},V) \\
& =2,
\end{split}\end{equation*}
the last equality following from $H^0(G_{K,S},V)=0$ and Tate's formula for the global Euler characteristic.
\end{proof}

Suppose that $p$ splits in $K$:
$$
p = \grp \bar\grp.
$$
Let
$$
H_\grp^1(K,V) = \ker\{H^1(G_{K,S},V)\stackrel{\res}{\rightarrow} H^1(K_{\grp},V)\},
$$
and let $H^1_{\bar\grp}(K,V)$ be defined similarly; these are independent of the finite
set $S$.

\begin{lem}\label{keylemma} If
$\dim_L\im\{H^1_f(K,V)\stackrel{\res}{\rightarrow} H^1_f(K_p,V)\}=1$, then
$$
H^1_\grp(K,V) = H^1_{str}(K,V) = H^1_{\bar\grp}(K,V).
$$
If also $\dim_L H^1_f(K,V)=1$, then $H^1_\grp(K,V) = H^1_{str}(K,V) = H^1_{\bar\grp}(K,V)=0$.
\end{lem}

\begin{proof}
By Lemma \ref{dim2lemma}, the image $X$ of $H^1(G_{K,S},V)$ in $H^1(K_p,V)$ is two-dimensional
over $L$. Let $X_\grp$ and $X_{\bar\grp}$ be the respective images of $H^1_\grp(K,V)$
and $H^1_{\bar\grp}(K,V)$ in $H^1(K_p,V)$; the action of the non-trivial automorphism $c$ of $K$ swaps $X_\grp$ and $X_{\bar\grp}$.  Let $X_f$ be the image of $H^1_f(K,V)$ in $H^1(K_p,V)$; this is stable under $c$ and one-dimensional by hypothesis.

Suppose $X_\grp\neq 0$. Then $X_{\bar\grp}\neq 0$ and $X = X_\grp\oplus X_{\bar\grp}$, from which it follows that $X=X^+\oplus X^-$ with
$X^\pm = \{ x\pm c(x) \ : \ x\in X_\grp\}$, and $X^\pm$ is one-dimensional. Then
$X_f$ equals $X^+$ or $X^-$. But it then follows that
$X_\grp\subset H^1_{f}(K_{\bar \grp},V)$ and $X_{\bar\grp}\subset H^1_{f}(K_{\grp},V)$
and hence that $X_f$ is two-dimensional,
a contradiction. Therefore, $X_\grp = 0 = X_{\bar\grp}$.

If also $\dim_L H^1_f(K,V) =1$, then $H^1_{str}(K,V)=0$, whence the final conclusion.
\end{proof}

\subsection{The Heegner points $P_K(f)$}\label{Heegnerpoints}
We consider two cases:
\begin{itemize}
 \item[I.] Every prime $\ell\mid N$ either splits or ramifies in $K$.
 \item[II.] $N=N^-N^+$ with $N^-$ a nontrivial product of an even number of primes that are inert in $K$
 and $N^+$ is the product of primes that split in $K$ (in particular, $(D,N)=1$).
\end{itemize}

Suppose first that $f$ and $K$ are as in Case I. Let $\T$ be the Hecke algebra generated over $\Z$ by the
usual Hecke operators $T_\ell$, for primes $\ell\nmid N$, acting on the
space of cuspforms $S_2(\Gamma_0(N))$. Let $\T_{M_f}=\T\otimes M_f$ and let $\veps_f\in \T_{M_f}$
be the idempotent corresponding to the projection $\T_{M_f}\twoheadrightarrow M_f$ sending
$T_\ell$ to the eigenvalue $a_\ell(f)$ of its action on the newform $f$.
Let $\grp_f=\ker\veps_f|_{\T}$.

Let $X=X_0(N)$ be the modular curve over $\Q$. Then $f$ determines a differential
$\omega_f\in \Omega^1(X)\otimes\C = \Omega^1(X(\C))$: the pullback of $\omega_f$ to the upper half-plane
via the usual complex uniformization of $X(\C)$ is $2\pi if(\tau)d\tau$. The operator $T_\ell$
can also be viewed as acting via a correspondence on $X$ such that $T_\ell\cdot\omega_f = \omega_{T_\ell\cdot f}= a_\ell(f)\cdot \omega_f$. The induced action of the $T_\ell$'s on
the Jacobian $J(X)$ of $X$ realizes $\T$ as a subring of $\End_\Q(J(X))$, 
and $\omega_f$ is a basis for the 
one-dimensional $M_f$ space $\veps_f(\Omega^1(J(X))\otimes M_f)$ (where we 
have identified $\Omega^1(X)=\Omega^1(J(X))$ in the usual way). The abelian variety $A_f$ is just the quotient $A_f = J(X)/\grp_f J(X)$, and 
we let $\phi:J(X)\rightarrow A_f$ be the quotient map. We let $\omega\in \Omega^1(A_f)\otimes M_f$ be the $1$-form such that
$\phi^*\omega = \omega_f$.

Let $\O_K$ be the integer ring of $K$ and $\grn\subset\O_K$ an ideal of norm $N$; this
is possible by the hypothesis that each prime $\ell\mid N$ either splits or ramifies in $K$. The
degree $N$ isogeny $\C/\O_K\rightarrow \C/\grn^{-1}$ (the canonical projection)
of CM elliptic curves is cyclic since $N$ is square free (in particular, $\ell^2\nmid N$ if $\ell$ ramifies in $K$) and 
defines a point $P\in X(H)$ on $X$ over the Hilbert class field $H$ of $K$.
Let $D_K= \sum_{\sigma\in\Gal(H/K)} P^\sigma \in \Div(X)$. For $\ell\nmid N$, let
$D_{K,\ell} = (T_\ell-1-\ell)\cdot D_K\in \Div^0(X)$ (the degree of the correspondence
$T_\ell$ is $\ell+1$). Let $Q_K(f)= \frac{1}{a_\ell(f)-1-\ell}\veps_f\cdot [D_{K,\ell}]
\in J(X)(K)\otimes M_f$. This is independent of $\ell$ since $\veps_f\cdot T_\ell = a_\ell(f)\cdot\veps_f$ in $\T_{M_f}$. Put 
$$
P_K(f) = \phi(Q_K(f))\in A_f(K)\otimes M_f.
$$

For the purposes of comparison with other constructions, we also consider
$D_K^0 = D_K-\#\Gal(H/K)\cdot\infty \in \Div^0(X)$, where $\infty\in X(\Q)$ is the usual cusp
at infinity, and $Q_K^0 = [D^0_K] = \sum_{\sigma\in\Gal(H/K)}[P-\infty]^\sigma \in J(X)(K)$.
As $T_\ell\cdot\infty = (1+\ell)\cdot \infty$ (that is, the cusps are Eisenstein), 
$\veps_f\cdot Q_K^0 = Q_K(f)$. 
Similarly, if $\xi\in \Div(X)\otimes\Q$ is the normalized, degree one Hodge divisor
defined in \cite[\S\S 1.2.2, 3.1.3]{YZZ}, we let $D_K^\xi = D_K-\#\Gal(H/K)\cdot\xi \in \Div^0(X)\otimes \Q$
and $Q_K^\xi = [D_K^\xi] = \sum_{\sigma\in \Gal(H/K)} 
[P-\xi]^\sigma\in J(X)(K)\otimes\Q$. It follows directly from the expression
in \cite[\S 3.1.3]{YZZ} for the Hodge divisor in terms of the canonical divisor that 
$\xi$ is also Eisenstein, so $\veps_f\cdot Q_K^\xi = Q_K(f)$.

In Case II we let $X$ be the Shimura curve over $\Q$ associated with the indefinite
quaternion algebra $B$ of discriminant $N^-$ and an Eichler order $\O_{B,N^+}$ of level
$N^+$ in a maximal order $\O_B$ of $B$. Let $K\hookrightarrow B$ be
an embedding such that $K\cap \O_{B,N^+} = \O_K$. Replacing $f$ with its
Jacquet-Langlands transfer $f_B$ to the space of weight 2 cuspforms for the subgroup
determined by $O_{B,N^+}$ (and normalized as in \cite[\S 2.8]{Hunter} to be defined over $M_f$)
and $\grn$ by an integral ideal $\grn^+$ of $K$ with norm $N^+$, 
there are constructions analogous to those yielding $Q_K(f)$ and $P_K(f)$ in Case I 
that in this case yield $Q_K(f) = Q_K(f_B)\in J(X)(K)\otimes M_f$ and $P_K(f)=P_K(f_B)\in A_f(K)\otimes M_f$.
We can also define $Q_K^\xi$ in this case, and, just as in Case I, $\veps_f\cdot Q_K^\xi = Q_K(f)$.

\subsection{The Gross--Zagier theorem}  We recall a consequence of a special case of the general Gross--Zagier formula of Yuan, Zhang, and Zhang
\cite{YZZ}. 

Consider the following hypotheses:
\begin{equation}
\text{$N$ is squarefree,}\tag{sqf}
\end{equation}
\begin{equation}
\eps(f,K)=-1.\tag{sgn}
\end{equation}
As the central character of $\pi$ is trivial, the first of these hypotheses implies that
for each $\ell\mid N$, $\pi_\ell$ is either $\sigma_\ell$ or $\sigma_\ell\otimes\xi_\ell$, where
$\sigma_\ell$ is the special representation and $\xi_\ell$ is the unramified quadratic character
of $\Q_\ell^\times$.
The second implies that that the number of primes $\ell\mid N$ for which $\eps_\ell(\pi,K)=-1$
is even. In particular, if $(D,N)=1$ then, since $\eps_\ell(\pi,K)=+1$ if $\ell$ splits in $K$
and $\eps_\ell(\pi,K)=-1$ if $\ell\mid N$ is inert in $K$, 
the number of prime divisors of $N$ that are inert in $K$ is even. 

Consider the additional hypothesis:
\begin{equation}
\text{If $(D,N)\neq 1$: $\pi_\ell\cong\sigma_\ell\otimes\xi_\ell$ for each
 $\ell\mid (D,N)$, and each $\ell\mid \frac{N}{(D,N)}$ splits in $K$.}\tag{ram}
 \end{equation}
If (sqf), (sgn), and (ram) hold, then $f$ and $K$ are as in either Case I or Case II
of Section \ref{Heegnerpoints}.

\begin{prop}[{\cite{YZZ}}]\label{YZZprop} Suppose 
{\rm (sqf)}, {\rm (sgn)}, and {\rm (ram)} hold. If $P_K(f)\neq 0$, then $\ord_{s=1}L(f,K,s)=1$.
\end{prop}

We explain how Proposition \ref{YZZprop} follows from the main result of \cite{YZZ}. 
Let $\iota:M_f\hookrightarrow \C$ be the identity map. We first note
that $P_K(f) = h_K \cdot P(f_1)^\iota$, where $h_K = \#\Gal(H/K)$ and
$P(f_1)\in A_f(K)$ is 
associated to the homomorphism $f_1=\phi:J(X)\rightarrow A_f$ and the point $P\in X(H)$ as in \cite[\S 3.2.5]{YZZ}. Let $\CL$ be a symmetric, ample line bundle on 
$A_f$ and $\lambda:A_f\rightarrow A_f^\vee$ the associated polarization. Then
$\lambda(P_K(f))\in A_f^\vee(K)\otimes M_f$ is just the point 
$h_K\cdot P(f_2)^\iota$ as in {\it loc.~cit.}~with $f_2 = \lambda\circ\phi:J(X)\rightarrow A_f^\vee$. Let $\langle\cdot,\cdot\rangle_\CL$
be the N\'eron-Tate height-pairing on $A_f(K)$ associated with the line bundle $\CL$, and let
$\langle \cdot,\cdot\rangle_{NT}$ be the canonical height pairing on 
$A_f(\overline K)\otimes \R \times A_f^\vee(\overline K)\otimes \R$. Then
\begin{equation*}\begin{split}
\langle P_K(f), P_K(f) \rangle_\CL = \langle P_K(f),\lambda(P_K(f)\rangle_{NT} & = h_K^2 \langle P(f_1)^\iota, P(f_2)^\iota\rangle_{NT} \\ 
&  = \frac{ h_K^2\zeta(2) L'(f,K,1) \cdot (f_1\circ f_2^\vee)^\iota}{4L(\chi_D,1)^2 L(\mathrm{Sym}^2 f,2) \mathrm{vol}(X)},
\end{split}\end{equation*}
where $f_1\circ f_2^\vee \in \End^0_\Q (A_f) = M_f$.  The last equality is just
\cite[Thm.~3.13]{YZZ}. As $P_K(f)\neq 0$ if and only if $\langle P_K(f),P_K(f)\rangle_\CL\neq 0$, the proposition follows.

If $P_K(f) = \sum P\otimes r_P$ with $P$ running over a basis of $A_f(K)\otimes\Q$
and $r_P\in M_f$,
then for any $\sigma:M_f\hookrightarrow \R$, $P_K(f^\sigma) = \sum P\otimes \sigma(r_P)$. So
$P_K(f)\neq 0$ if and only if $P_K(f^\sigma)\neq 0$ for all $\sigma$. Appealing to the above
proposition for all the Galois conjugates $f^\sigma$ of $f$ we deduce:

\begin{coro}\label{GZcoro}
If $P_K(f)\neq 0$, then $\ord_{s=1}L(A_f/K,s) = [M_f:\Q]$.
\end{coro}

\subsection{A $p$-adic $L$-function and a formal log of $P_K(f)$}
As described in the introduction, our proof that $P_f(K)\neq 0$, and 
hence of Theorem \ref{thmB}, hinges on
an identity expressing the value of a certain $p$-adic $L$-function as a non-zero
multiple of the square of a formal logarithm of $P_K(f)$. We now recall
this $L$-function and identity.

Let $K_\infty/K$ be the anticyclotomic $\Zp$-extension of $K$; so $\Gamma=\Gal(K_\infty/K)\cong \Z_p$
and conjugation by $c$ sends
$\gamma\in \Gamma$ to $\gamma^{-1}$. Let $\Psi:G_K\twoheadrightarrow\Gamma$ be the canonical projection. We continue to assume that $p=\grp\bar\grp$ splits in $K$.

We also assume that $\lambda$ is the prime of $M_f$ determined by the fixed embedding $\bQ\hookrightarrow \overline K_\grp$. And for the purposes of $p$-adic interpolation and comparisons with complex values, we fix an embedding $\overline K_{\grp}\hookrightarrow\C$ such that induced complex embedding of $\bQ$ is just the fixed one.

 Given  a continuous character $\psi:\Gamma\rightarrow\Q_p^\times$ we consider $\psi$ to be a Galois character via composition with the projection $\Psi$.
We say that such a $\psi$ is Hodge-Tate
if $\psi$ is Hodge-Tate as a representation of both $G_{K_\grp}$ and
$G_{K_{\bar\grp}}$. As $\psi$ is anticyclotomic (that is, $\psi(c^{-1}gc)= \psi(g)^{-1}$),
$\psi$ is Hodge-Tate if and only if it is Hodge-Tate as a representation of one of
$G_{K_\grp}$ and $G_{K_{\bar\grp}}$; the respective Hodge-Tate weights must be
$-n$ and $n$ for some integer $n$, and in this case we say that $\psi$ is Hodge-Tate
of weight $(-n,n)$. The Galois character $\psi$ is the $p$-adic avatar
of a unitary algebraic Hecke character $\psi^{\alg}$ of $K$ with infinity type $z^{n}\bar z^{-n}$.
The character $\psi^{\alg}:\A_K^\times\rightarrow \C^\times$ is given by
$\psi^\alg((x_v)) = x_\grp^{-n}x_{\bar\grp}^{n}x_\infty^{n}\bar x_\infty^{-n} \cdot\psi\circ\rec_K((x_v))$, where $\rec_K:K^\times\backslash\A_K^\times\rightarrow G_K^{ab}$ is the reciprocity map of class field theory, which we normalize so that uniformizers correspond to geometric Frobenius elements. In particular, there is an equality of
$L$-functions
$L(\psi^{\alg},s) = L(\psi,s)$.  The characters $\psi$ and $\psi^{\alg}$ are unramified at all places not dividing $p$.  
Let $\Sigma^{c}_\grp$ be the set of crystalline characters $\psi:\Gamma\rightarrow \Q_p^\times$ of
weight $(-n,n)$ with $n>0$ and $n\equiv 0 \, \mathrm{mod} \ p-1$; the crystalline condition is equivalent to $\psi^{\alg}$ being unramified at $\grp$ and $\bar\grp$.

Suppose in addition to (sqf), (sgn), and (ram) that
\begin{equation}
p\nmid N, \tag{flt}
\end{equation}
\begin{equation}
\text{$D$ is odd,}\tag{odd}
\end{equation}
\begin{equation}
\text{$L$ contains a large enough}\footnote{It is enough that $L$ contain the image of the Hilbert class field of $K$, though this is
not important here.}\text{ finite
extension of $\Qp$.}\tag{$L$-lrg}
\end{equation}
Let $S=\{\ell\mid pND\}$, and let $\O$ be the ring of integers of $L$.
There is an anticyclotomic $p$-adic $L$-function
$L^S_\grp(f)\in \O[\![\Gamma]\!]$ such that for $\psi\in\Sigma_\grp^{c}$
\begin{equation*}\label{acpadicL}
\psi(L^S_\grp(f)) =  C(f,K) w(f,\psi)e_\infty(f,\psi)e_p(f,\psi)\Omega_p(\psi)\frac{L^{S}(f,\psi^{\alg},1)}{\Omega(\psi^{\alg})},
\end{equation*}
where
\begin{itemize}
\item $L^S(f,\psi^{\alg},s) = L^S(V^\vee\otimes\psi,s)$ is the $L$-function with the Euler factors at the primes of $K$ dividing the primes in $S$ omitted;
\item $\Omega(\psi^{\alg}) =  (2\pi i)^{-2n-1}(2\pi)^{1-2n}\Omega^{4n}D^{n-1/2}$, with $\Omega$ a period of an elliptic curve 
$E_0$ with CM by the ring of integers of $K$;
\item $\Omega_p(\psi) = \Omega_p^{4n}$ with $\Omega_p$ a $p$-adic period for the elliptic curve $E_0$;
\item $e_\infty(f,\psi) = 4(2\pi)^{-2n-1}\Gamma(n)\Gamma(n+1)$;
\item $e_p(f,\psi) = \frac{(1-a_p(f)\psi^{\alg}(\bar\grp)p^{-1} + \psi^\alg(\bar\grp)^2p^{-1})}
{(1-a_p(f)\psi^{\alg}(\grp)p^{-1}+\psi^{\alg}(\grp)^2p^{-1})}$, where $a_p(f)$ is the
eigenvalue for the action of the Hecke operator $T_p$ on $f$;
\item $w(f,\psi) = \psi(W)$ for a unit $W\in \O[\![\Gamma]\!]^{\times}$ and
$w(f,1) = w_K$, the number of roots of unity in $K$;
\item $C(f,K)$ is some non-zero constant depending on $f$ and $K$.
\end{itemize}
For any $\psi$ we write $L_\grp^S(f,\psi)$ for
$\psi(L_\grp^S(f))$. 

If there exists a prime $\ell\mid N$ that is inert or ramified in $K$, then this $p$-adic $L$-function is constructed in \cite{EHLS} and in \cite{Wan-U31}.
It has been constructed more generally in \cite{BDP}, \cite{Hunter}, and \cite{MB11}. It can be related
to a specialization of a three-variable $L$-function constructed by Hida \cite{Hi-padicRankinL} (this is done in \cite{Wan-U31}).
To be precise, the functions constructed in \cite{BDP} and \cite{Hunter} are {\it a priori} only continuous on $\Gamma$. That they belong
to the Iwasawa algebra (which follows from the constructions in \cite{EHLS} and \cite{Wan-U31}) requires additional argument, essentially
extending the formulas to characters ramified at primes above $p$ as is done in \cite{MB11}.

For $\psi\in \Sigma_\grp^c$, let $\chi$ be the Hecke character of $K$ such that $\chi^{-1}= \psi^\alg|\Nm(\cdot)|_\Q$, where $\Nm$ is the norm map from $\A_K$ to $\A_\Q$ and 
$|\cdot|_\Q$ is the usual absolute value on $\A_\Q$.  Recall that under the assumptions
(sqf), (sgn), and (ram), $f$ and $K$ are as in Case I or Case II of Section \ref{Heegnerpoints}.
Then $L_\grp^S(f)$ is the
imprimitive variant of the $p$-adic $L$-function denoted $L_p(f,\chi)$ in \cite{BDP} (in Case I) and
in \cite{Hunter} (in Case II).  By `imprimitive variant' we mean that
the Euler factors at the primes in $S$ not dividing $p$ have been removed.
The set  of such $\chi$ for $\psi\in\Sigma_\grp^c$ is denoted $\Sigma_{cc}^{(2)}(\grn)$
and $\Sigma_{cc}^{(2)}(\grn^+)$, respectively, in \cite{BDP} and \cite{Hunter}, and
the interpolated values are given in terms of the values $L(f,\chi^{-1},0)$, which is just
$L(f,\psi^\alg,1)$. Also, still in the notation of \cite{BDP}
and \cite{Hunter}, $e_\infty(f,\psi)w_K(2\pi i)^{1+2n}=C(f,\chi,1)$, and
$w(f,\psi) = w_K w(f,\chi)^{-1}$. The constant we have denoted $C(f,K)$ is denoted
$\alpha(f,f_{\GL_2})^{-1}$ in \cite{Hunter} (wherein $f$ denotes a form on 
an indefinite quaternion algebra and $f_{\GL_2}$ is a suitably normalized Jacquet-Langlands lift of $f$
to $\GL_2$; the quaternion algebra depends on $K$). 

Among the important results in \cite{BDP2} and \cite{Hunter} is the following
expression for the value of $L^S_\grp(f,\psi)$ at the trivial character $\psi=1$, which
we call the BDP point, in terms of a formal log of the  point $P_K(f)$.
Note that the BDP point is not in $\Sigma^{c}_\grp$.

\begin{prop}[{\cite{BDP2}}, \cite{Hunter}]\label{BDPprop} 
Suppose $(\mathrm{sqf})$, $(\mathrm{sgn})$, $(\mathrm{ram})$, $(\mathrm{flt})$, 
and $(\mathrm{odd})$ hold. Then
\begin{equation*}\label{BDPformula}
L^S_\grp(f,1)\stackrel{.}{=} (\log_\omega P_K(f))^2.
\end{equation*}
\end{prop}

\noindent  Recall that `$\stackrel{.}{=}$' means equality up to a non-zero constant
(which here also depends on the eigenvalues at the primes in $S$), and 
$\omega\in \Omega^1(A_f)\otimes M_f$ is a $1$-form such that $\phi^*\omega = \omega_f$
(so the action of $M_f$ on $\omega$ through its action on $\Omega^1(A_f)$ agrees with its scalar action),
and $\log_\omega:A_f(K_\grp)\otimes L\rightarrow \overline K_\grp$ is the
formal logarithm determined by $\omega$.

In Case I, Proposition \ref{BDPprop} is just \cite[Thm.~3.12]{BDP2}. The formula in 
{\it loc.~cit.}~has $\log_\omega P_K(f)$ replaced with $\log_\omega P_f$, where
$P_f=\phi(Q_K^0)$. But $\log_\omega P_K = \log _{\omega_f} Q_K^0$,
and, since $\veps_f\cdot \omega_f = \omega_f$, $\log_{\omega_f} Q_K^0 = 
\log_{\omega_f} \veps_f\cdot Q_K^0 = \log_{\omega_f} Q_K(f) = \log_\omega P_K(f)$,
whence the formula in this case.  In Case II, the proposition similarly follows from 
\cite[Prop.~8.13]{Hunter}.

To prove that $P_K(f)\neq 0$, and hence that $A_f(K)$ has positive rank,
it suffices to show that $\log_\omega P_K(f)\neq 0$ and, therefore,
to show that $L^S_\grp(f,1)\neq 0$.

\begin{coro}\label{Heegner-padic}
$P_K(f)\neq 0$ if and only if $L_\grp^S(f,1)\neq 0$.
\end{coro}

In the following section we explain some other consequences of $L_\grp^S(f,1)=0$.

\subsection{Some Iwasawa theory for $f$ and $K$}
\label{someIwasawa}
We continue to assume $p=\grp\bar\grp$ splits in $K$.

Let $\O$ be the ring of integers of $L$, and let $T\subset V$ be a $G_\Q$-stable $\O$-lattice.
Let $\Lambda=\O[\![\Gamma]\!]$. We view the projection $\Psi:G_K\twoheadrightarrow \Gamma$ as a continuous $\Lambda^\times$-valued character. Let
$\Lambda^*=\Hom_{cts}(\Lambda,\Qp/\Zp)$ be the Pontryagin dual of $\Lambda$, and let
$$
M = T\otimes_\O\Lambda^*(\Psi^{-1}),
$$
that is, the discrete $\Lambda$-module $T\otimes_\O\Lambda^*$ with continuous $G_K$
action $\rho_{f,\lambda}\otimes\Psi^{-1}$.
Let $S=\{\ell\mid pND\}$.
Let
\begin{equation*}\begin{split}
\Sel_\infty(f,K,S) 
= \ker\left\{H^1(G_{K,S},M)\stackrel{\res}\rightarrow H^1(I_\grp,M)\right\}.
\end{split}\end{equation*}
This is a discrete $\Lambda$-module, and its Pontryagin
dual
$$
X_\infty(f,K,S) = \Hom_{\Lambda}(\Sel_\infty(f,K,S),\Lambda^*)
$$
is a finitely generated $\Lambda$-module. Let $Ch_\Lambda(f,K,S)$ be its characteristic ideal over $\Lambda$; this is non-zero
if and only if $X_\infty(f,K,S)$ is a torsion $\Lambda$-module.

The Selmer group $\Sel_\infty(f,K,S)$ is essentially an imprimitive version of one of Greenberg's Selmer groups for the `big' Galois
module $M$, as we now explain. Let $\psi\in\Sigma_\grp^c$ with Hodge-Tate weights $(-n,n)$ (recall $n>0$). 
Then the induced representation $\Ind_{G_K}^{G_\Q} (V\otimes\psi^{-1})$ satisfies the Panchishkin condition as in \cite[\S3]{Gr-mot}:
its restriction to $G_\Qp$ is just $(V\otimes\psi^{-1})\oplus (V\otimes\psi^{-c})$, where the first summand is identified with $(V\otimes \psi^{-1})|_{G_{K_\grp}}$
and has non-negative Hodge-Tate weights $n$ and $n-1$ while the second summand is identified with 
$(V\otimes \psi^{-1})|_{G_{K_{\bar\grp}}}$ and 
has negative Hodge-Tate weights $-n$ and $-n-1$; the subspace with negative\footnote{Our conventions for Hodge-Tate weights are the negative of those in \cite{Gr-mot}.} Hodge-Tate weights has dimension equal to the
dimension of the $+1$-eigenspace for complex conjugation on the induced represention. Furthermore, 
if $\gamma\in \Gamma$ is a topological generator then $M[\gamma-\psi(\gamma)] \cong (T\otimes \psi^{-1})\otimes_\Zp \Qp/\Zp$.
That is, $M$ is the discrete Galois module associated with the deformation $T\otimes\Lambda(\Psi^{-1})$ of $T$ over $\Lambda$,
which contains a Zariski-dense set of specializations that satisfy the Panchishkin condition (namely, the specializations
under the maps $\gamma\mapsto\psi(\gamma)$ for $\psi\in \Sigma_\grp^c$).  This is a simple generalization to $G_K$-representations
of the situation considered in \cite[\S4]{Gr-mot}, and the Selmer group $\Sel_\infty(f,K,S)$ is the corresponding generalization 
of the Selmer groups defined in {\it loc.~cit.} Analogously to \cite[Conj.~4.1]{Gr-mot}, assuming also
\begin{equation}
\text{$\bar\rho_{f,\lambda}|_{G_K}$ is irreducible,}\tag{irr${}_K$}
\end{equation}
one then conjectures\footnote{The order of the Selmer group for $V\otimes \psi^{-1}$ is expected to be related to the
$L$-value $L(V^\vee\otimes \psi,1)$. This dictates which $p$-adic $L$-function should be identified with the characteristic ideal of 
$\Sel_\infty(f,K,S)$, namely, $L_\grp^S(f)$.} :
\begin{conj}\label{ACMC} 
$Ch_\Lambda(f,K,S) = (L_\grp^S(f))$ in $\Lambda\otimes_\Zp\Qp$.
\end{conj}
\noindent This is the Iwasawa--Greenberg Main Conjecture for $M$.

Significant progress toward this conjecture has been made by X. Wan \cite{Wan-U31}, following the methods of \cite{SU-MCGL}. 
Suppose in addition to (sqf), (flt), ($L$-lrg), and (irr${}_K$) that (sgn) and (odd) hold and that
\begin{equation}
p\geq 5,\tag{big}
\end{equation}
\begin{equation}
\text{$f$ is ordinary with respect to $\lambda$,}\tag{ord}
\end{equation}
\begin{equation}
\text{both $2$ and $p$ split in $K$,}\tag{spl}
\end{equation}
\begin{equation}
\text{$\bar\rho_{f,\lambda}$ is ramified at some odd prime $\ell|N$ that is inert or ramified in $K$}.\tag{res}
\end{equation}
Then it is proved in \cite{Wan-U31} that:
\begin{prop}[\cite{Wan-U31}]\label{Wanprop} Under the above assumptions,
$$ Ch_\Lambda(f,K,S) \subseteq (L^S_\grp(f))$$ in $\Lambda\otimes_\Zp\Qp$.
\end{prop}
\noindent This proposition is not explicitly given in \cite{Wan-U31}, however it follows easily from part (2) of 
\cite[Thm.~1.1]{Wan-U31}, as we explain. 

Let $\pi$ be the cuspidal automorphic representation of $\GL_2(\A_\Q)$
such that $L(\pi,s-1/2) = L(f,s)$ and let $f_0$ be the ordinary stabilization of the newform $f$. 
We fix a character $\psi_0\in \Sigma_\grp^c$ with Hodge-Tate weights
$(-n_0,n_0)$ with $n_0>6$.  Let $\psi^\alg$ be the associated algebraic Hecke character and let
$\xi = \psi^\alg\omega$, where $\omega$ is the finite order Hecke character of $K$ associated
with the Teichm\"uller character (the lift of the mod $p$ reduction of the $p$-adic cyclotomic character). 
Then $\pi$, $f_0$, and $\xi$ satisfy the hypotheses of part (2) of \cite[Thm.~1.1]{Wan-U31} (in the notation of
{\it loc.~cit.}, $\Gamma_\CK$ is the Galois group of the composite of all $\Zp$-extensions of $\CK=K$, so
$\Gamma_\CK \cong \Z_p^2$ and $\Gamma$ is a natural quotient of $\Gamma_\CK$). It remains to explain
how the conclusion of that part of the theorem (or really its proof) implies the above proposition.
The $p$-adic $L$-function $\CL_{f_0,\CK,\xi}^{\Sigma,Hida} \in \O[\![\Gamma_\CK]\!]$ with $\Sigma = S$ from
\cite[\S6]{Wan-U31} is a two-variable extension
of the $p$-adic $L$-function considered herein: under the composition
$$
\O[\![\Gamma_\CK]\!]\stackrel{\Gamma_\CK\twoheadrightarrow\Gamma}{\longrightarrow} \Lambda \stackrel{\gamma\mapsto \psi_0(\gamma)^{-1}\gamma}{\longrightarrow} \Lambda,
$$
$\CL_{f_0,\CK,\xi}^{\Sigma,Hida}$ maps to $L_\grp^S(f)$. Similarly, under the base change from $\O[\![\Gamma_\CK]\!]$ to
$\Lambda$  given by this map (that is, tensoring with $\Lambda$ over $\O[\![\Gamma_\CK]\!]$), 
the Selmer group denoted $\Sel^\Sigma_{f_0,\CK,\xi}$ in \cite[\S2.11]{Wan-U31} becomes
$\Sel_\infty(f,K,S)$ and $X^\Sigma_{f_0,\CK,\xi}$ becomes $X_\infty(f,K,S)$. 
The proof of part (2) of \cite[Thm.~1.1]{Wan-U31} involves showing that for $\Sigma = S$, 
the characteristic ideal $\mathrm{char}_{\O[\![\Gamma_\CK]\!]}X^\Sigma_{f_0,\CK,\xi}$ is contained in $(\CL^{\Sigma,Hida}_{f_0,\CK,\xi})$
in $\O[\![\Gamma_\CK]\!]\otimes_\O L$.
The corresponding inclusion $Ch_\Lambda(f,K,S) \subset (L_\grp^S(f))$ in $\Lambda\otimes_\Zp\Qp$ then follows easily, 
using that the rings $\O[\![\Gamma_\CK]\!]$ and $\Lambda$ are unique factorization ideals and the various 
characteristic ideals are principal (cf.~\cite[Cor.~3.8]{SU-MCGL}).

We now explain a simple consequence of $L_\grp^S(f,1)=0$. Let $W = M[\gamma-1] = T\otimes_\Zp\Qp/\Zp$ and let
\begin{equation*}\begin{split}
\Sel_\grp(f,K,S) 
= \ker\left\{H^1(G_{K,S},W) \stackrel{\res}{\rightarrow} H^1(I_\grp,W)\right\}.
\end{split}
\end{equation*}
Let also
$$
X_\grp(f,K,S) = \Hom_\Zp(\Sel_\grp(f,K,S),\Qp/\Zp).
$$
Assuming (irr${}_K$), 
$$
H^1(G_{K,S},W)=H^1(G_{K,S},M[\gamma-1])= H^1(G_{K,S},M)[\gamma-1].
$$
It then follows from the exactness of the bottom row of the commutative diagram
$$
\begin{CD}
{} @. H^1(G_{K,S},W) @= H^1(G_{K,S},M)[\gamma-1] \\
{} @. @VresVV @VresVV \\
(M^{I_\grp}/(\gamma-1)M^{I_\grp})^{G_{K_\grp}} @>>>
H^1(I_\grp,W)^{G_{K_\grp}} @>>> H^1(I_\grp,M)[\gamma-1]^{G_{K_\grp}}
\end{CD}
$$
\noindent (the bottom row comes from the long exact cohomology sequence associated with 
the short exact sequence $0\rightarrow W \rightarrow M \stackrel{\times (\gamma-1)}{ \rightarrow} M \rightarrow 0$) that
$\Sel_\grp(f,K,S)$ is contained
in $\Sel_\infty(f,K,S)[\gamma-1]$ with finite index, and therefore
\begin{equation}\label{SelmerK-iso}
\# X_\grp(f,K,S) = \infty \iff \#(X_\infty(f,K,S)/(\gamma-1)X_\infty(f,K,S)) = \infty.
\end{equation}
Suppose $L_\grp^S(f,1)=0$. This means
$L_\grp^S(f) \in (\gamma-1)$, so if the assumptions of Proposition \ref{Wanprop} also hold, then $Ch_\Lambda(f,K,S)\subset (\gamma-1)$. 
By basic properties of characteristic ideals, this last inclusion implies $\#(X_\infty(f,K,S)/(\gamma-1)X_\infty(f,K,S)) = \infty$.
Combining this with \eqref{SelmerK-iso} we conclude
\begin{equation}\label{keyeq1}
L_\grp^S(f,1) = 0 \implies \# X_\grp(f,K,S) = \infty.
\end{equation}
As $H^1(K_\grp,V)\hookrightarrow H^1(I_\grp,V)$,
we have $H^1_\grp(K,V) = \ker\{H^1(G_{K,S},V)\rightarrow H^1(I_\grp,V)\}$,
from which it follows easily that the image of $H^1_\grp(K,V)$ in $H^1(G_{K,S},W)$ is the maximal divisible subgroup of $\Sel_\grp(f,K,S)$.
Combining this observation with \eqref{SelmerK-iso} and \eqref{keyeq1}, we conclude that:

\begin{prop}\label{Heegner-Selmer} If the hypotheses of Proposition \ref{Wanprop}
hold, then 
$$
L^S_\grp(f,1)=0 \ \implies \ H^1_\grp(K,V)\neq 0.
$$
\end{prop}

\subsection{An observation about hypothesis (irr${}_K$)}
We include a simple lemma on the irreducibility of $\bar\rho_{f,\lambda}|_{G_K}$. Consider the hypothesis
\begin{equation}
\text{$\bar\rho_{f,\lambda}$ is an irreducible $G_\Q$-representation.}
\tag{irr}
\end{equation}
\begin{lem}\label{irr-lem} 
If {\rm (irr)} and {\rm (res)} hold, then so does {\rm (irr${}_K$)}.
\end{lem}

\begin{proof}
Suppose (irr) and (res) hold. Let $\ell\mid N$ be a prime at which $\bar\rho_{f,\lambda}$ is ramified. 
As $\ell\mid\mid N$, the action of $I_\ell$ on $V$ is unipotent and factors through tame inertia. In particular,
if $\tau_\ell$ is a topological generator of tame inertia at $q$, then $\rho_{f,\lambda}(\tau_\ell)$ is unipotent
\footnote{This follows from the local-global compatibility satisfied by $\rho_{f,\lambda}|_{G_{\Q_\ell}}$.},
hence so, too, is $\bar\rho_{f,\lambda}(\tau_\ell)$; the latter is a unipotent element of order a power of $p$.
As $\tau_\ell^2\in G_K$, it follows that the image of $\bar\rho_{f,\lambda}|_{G_K}$ contains a unipotent element
of order a power of $p$. Let $k$ be the residue field of $L$. Suppose now that $\bar\rho_{f,\lambda}$ is reducible over $\bar k$. 
Then the image of $\bar\rho_{f,\lambda}|_{G_K}$ is contained in either a torus (split or non-split) or a Borel of $\GL_2(k)$. 
The first possibility is ruled out as the image contains a unipotent element of order a power of $p$; the image of
$\bar\rho_{f,\lambda}|_{G_K}$  is therefore contained in a Borel. 
But as the image of $G_K$ is normalized by the image of $G_\Q$, it follows easily that the image of $G_\Q$ is also 
contained in the Borel, contradicting (irr). 
\end{proof}

\subsection{Finishing the proof of Theorem \ref{thmB}}

Let $f$, $p$, and $K$ be as in Theorem \ref{thmB}.

We begin by noting that (sqf), (odd), and (big) hold by hypothesis.
Hypothesis (a) of Theorem \ref{thmB} ensures that (flt) and (ord) hold. Hypothesis (b) is just (irr) and (res) and so
(irr${}_K$) holds by Lemma \ref{irr-lem}, and hypothesis (c) is
just (spl). Hypothesis (e) then implies, by Lemma \ref{signlemma}, that (sgn) holds.
So Propositions \ref{Wanprop} and \ref{Heegner-Selmer} apply: if $H^1_\grp(K,V)=0$
then $L_\grp^S(f,1)\neq 0$. Hypothesis (e) also implies, by Lemma \ref{keylemma}, that
$H^1_\grp(K,V)=0$. Since hypothesis (d) is just (ram),
it follows from Corollary \ref{Heegner-padic} that $P_K(f)\neq 0$.

The rank of $A_f(K)$ is equal to $[M_f:\Q]$ times the $M_{f,p}$-rank of the free $
M_{f,p}$-space $A_f(K)\otimes\Qp$, which, by the injection
$A_f(K)\otimes\Qp\hookrightarrow H^1_f(K,V_pA_f)$,
is at most the $L$-dimension of $H^1(K,V_pA_f)\otimes_{M_{f,p}}L = H^1_f(K,V)$. The latter has
$L$-dimension $1$
by hypothesis (e). Since $P_K(f)\neq 0$, so $A_f(K)\otimes\Q\neq0$, it follows that $\rank_\Z A_f(K) = [M_f:\Q]$. It also follows, by Proposition \ref{YZZprop} and Corollary \ref{GZcoro}, that
$\ord_{s=1}L(f,K,s)=1$ and $\ord_{s=1}L(A_f/K,s)=[M_f:\Q]$.

To conclude that $\Sha(A_f/K)$ is finite, we first observe that it is enough to show that
both $\Sha(A_f)$ and $\Sha(A_f^D)$ are finite. To show these are finite we begin by noting that
since $L(f,K,s)=L(f,s)L(f\otimes\chi_D,s)$ has order $1$ at $s=1$,
one of $L(f,s)$ and $L(f\otimes\chi_D,s)$ has order $1$ at $s=1$ and the other has order $0$.
The finiteness of the Tate--Shafarevich groups then just follows from the work of
Gross, Zagier, and Kolyvagin, as cited in the introduction.

This completes the proof of Theorem B.

\begin{rmk}\label{rmksonhyps} We indicate how the various hypotheses
of Theorem \ref{thmB} intervene in its proof and make some additional remarks on the theorem and its proof.
\begin{itemize}
\item[(i)] The requirement that $N$ be squarefree is made in \cite{Wan-U31} and
in \cite{Hunter} (that $N$ be squarefree at those primes dividing
$(D,N)$ is also required in \cite{BDP} and \cite{BDP2}). 
\item[(ii)] The hypothesis that $p\geq 5$ comes from \cite{Wan-U31}, where it is
imposed for convenience.
\item[(iii)] The hypothesis that $D$ be odd is made in \cite{BDP} and \cite{Hunter} 
as well as \cite{Wan-U31} (in which 2 is also required to split in $K$) and stems 
from some gaps in our knowledge of the theta correspondence 
for local fields of residue characteristic 2.
\item[(iv)] The assumption that $p\nmid N$ intervenes most crucially in \cite{BDP}, \cite{BDP2}, 
and \cite{Hunter}. We have 
also used it to simplify our use of the parity conjecture \cite{Ne-parity} 
for Selmer groups of modular forms (to verify (sgn)).
\item[(v)] The hypothesis that $f$ is ordinary at some $\lambda\mid p$ is only needed to use
the results of \cite{Wan-U31} and, again, in our appeal to \cite{Ne-parity}. In particular,
$f$ being ordinary is not crucial for the methods employed herein: if a version of
Proposition \ref{Wanprop} were available\footnote{Such a result has been announced in a recent preprint of Wan.} 
for forms with finite non-critical slope,
for example, then all the results of this paper would hold in that case, provided the corresponding
Selmer parity result was also known (this is known for elliptic curves with supersingular reduction).
\item[(vi)] The hypothesis that $\bar\rho_{f,\lambda}$ is irreducible is required
in \cite{Wan-U31}, as is the hypothesis that $\bar\rho_{f,\lambda}$ is ramified
at an odd prime $\ell\neq p$ that is either inert or ramified in $K$. The latter ensures,
among other things, that $\pi$ has a transfer to a definite unitary group $U(2)$ (that is ramified at $\ell$) defined using $K$; 
the primary results in \cite{Wan-U31} relate the $p$-adic $L$-function $L_\grp^S(f)$ to the
index of an Eisenstein ideal on $U(3,1)$ coming from an Eisenstein series induced from this cuspform on $U(2)$.
\item[(vii)] The requirement that $p$ split in $K$ is needed to use the results
in \cite{BDP}, \cite{Hunter}, and \cite{Wan-U31}. It also comes into the Galois
arguments, especially the proof of Lemma \ref{keylemma}. 
It is likely one of the most difficult hypotheses to relax in the methods employed in this paper. 
\item[(viii)] The hypotheses in (d) for when $(D,N)\neq 1$ are needed to appeal to the
results of \cite{BDP}, which requires that $\eps_\ell(f,K)=+1$ for all primes 
$\ell\mid (N,D)$.
\item[(ix)] The hypothesis that $H^1_f(K,V)$ is one-dimensional is used to know
beforehand that the root number $\eps(f,K)$ is $-1$; that is, (sgn) holds. This is
required for the results in \cite{YZZ}, \cite{BDP}, \cite{BDP2}, and \cite{Hunter}.
\item[(x)] The injectivity of the restriction map 
$H_f^1(K,V)\stackrel{\res}{\rightarrow} \prod_{\grl\mid p} H^1(K_\grl,V)$ is needed to 
ensure that $H^1_\grp(K,V)=0$. Conjecturally, it should be enough that $H^1_f(K,V)$ is
one-dimensional, and then the injectivity would follow from the conclusion that
$A_f(K)\neq 0$. In \cite{Z-Koly} Wei Zhang obtains a version of Theorem \ref{thmB}
without requiring this injectivity, but at the expense of requiring certain Tamagawa numbers 
be indivisible by $p$.
\item[(xi)] Many, if not all, of the local hypotheses on $\pi$ and $K$ can likely be relaxed.
For example, recent work of Y.~Liu, S.~Zhang, and W.~Zhang essentially establishes 
the identity in Proposition \ref{BDPprop} in the general Gross--Zagier set-up of 
\cite{YZZ} (including over a totally real field).
\item[(xii)] As recalled in the introduction, the analog of Theorem \ref{thmA} for CM elliptic 
curves is explained in \cite[Thms.~8.1,8.2]{Ru-CMrational} as a consequence of 
Rubin's proof of the main conjecture for CM curves, Perrin-Riou's $p$-adic Gross--Zagier formula,
and Bertand's proof of the non-degeneracy of the relevant $p$-adic height pairing. 
It is also possible to give a proof for the CM case along the lines of the proof of Theorem \ref{thmB}
by using the Main Conjecture for CM forms and the analog of Proposition \ref{BDPprop} for 
the CM case, which is just \cite[Thm.~9.5]{Ru-CMrational} or \cite[Thm.~2]{BDP2}.
\item[(xiii)] The methods employed to prove Theorem \ref{thmB} in this paper can be adapted to provide an alternative
proof of the base case of the induction argument in \cite{Z-Koly} that avoids appealing
to \cite{SU-MCGL} and so should also work for supersingular primes (see remark (v)).
This is part of forthcoming work.
\end{itemize}
\end{rmk}

\section{Theorem A follows from Theorem B}\label{AimpliesB}

Let $f$ and $A_f$ be as in Theorem \ref{thmA}.  In particular, $f$ is not a CM form. 

Consider the set
of primes $p\nmid N$ that are greater than $4$ and unramified in $M_f$.
Suppose that $f$ is not ordinary for all $\lambda\mid p$ for some such $p$. Then
the norm of $a_p(f)$, which has absolute value at most $(2p^{1/2})^{[M_f:\Q]}$ by the
Ramanujan bounds, is an integer divisible by
$p^{[M_f:\Q]}$, so $a_p(f)=0$. But, as $f$ is not a CM form, 
the set of primes with $a_p(f)=0$ has density zero \cite[\S7.2, Cor.~2]{Serre-CDT}. Thus the set of primes $p\nmid N$
such that $f$ is ordinary with respect to some $\lambda\mid p$ of $M_f$ has density $1$.

If $p$ is sufficiently large, then $\bar\rho_{f,\lambda}$ is irreducible for all $\lambda\mid p$ 
\cite[Thm.~2.1]{Ribet-irred}. If for some $\ell\mid N$ there were arbitarily large primes $p$ and primes $\lambda\mid p$ of $M_f$
such that $\bar\rho_{f,\lambda}$ were unramified at $\ell$, then, by the finiteness of the
number of newforms of weight $2$ and level dividing $N$ and by the main result
of \cite{Ribet-level},
there would be a newform $g$ of weight $2$ and level prime to $\ell$ such that $f$
and $g$ would be congruent modulo primes of arbitrarily large characteristic $p$,
in the sense that their prime-to-$Np$ coefficients would be congruent. It would
then follow that the prime-to-$Np$ coefficients of $f$ and $g$ would be the same and hence, by
multiplicity one, that $f=g$, a contradiction. Thus, for sufficiently large $p$,
$\bar\rho_{f,\lambda}$ is ramified at all primes that divide $N$. 

By the preceding observations, we may fix a $p\geq 5$, $p\nmid N$, and a $\lambda\mid p$ 
such that $A_f$ is ordinary with respect to $\lambda$ and
$\bar\rho_{f,\lambda}$ is irreducible and ramified at all primes that divide $N$.
The hypotheses that $\rank_\Z A_f(\Q)=[M_f:\Q]$ and $\Sha(A_f)[p^\infty]$ are finite imply,
by the obvious analog of Lemma \ref{rank1lemma} with $K$ replaced by $\Q$, that $\dim_L H^1_f(\Q,V)=1$.
It then follows from Nekov\'a\v{r}'s work on the parity conjecture for Selmer groups of modular
forms \cite[Thm.~12.2.3]{Ne-parity} that $\eps(f)=-1$.

We choose an imaginary quadratic field $K/\Q$ of discriminant $D$ such that
\begin{itemize}
 \item[(i)] $2$ and $p$ split in $K$ (so $D$ is odd and hypothesis (c) of Theorem \ref{thmB} holds);
\item[(ii)] if for some odd prime $\ell$ the local representation $\pi_\ell$ is the twist of the special representation by the 
unique unramified quadratic character,
then $K$ is ramified at $\ell$ but all other prime divisors of $N$ split in $K$;
\item[(iii)] if there is no $\pi_\ell$ as in (ii) but there are two odd primes $\ell_1$ and $\ell_2$ such that
$\pi_{\ell_1}$ and $\pi_{\ell_2}$ are special, then $\ell_1$ and $\ell_2$ are inert in $K$ and
all other prime divisors of $N$ split in $K$;
\item[(iv)] $L(f\otimes\chi_D,1)\neq 0$.
\end{itemize}
If (i) and (ii) hold, then $\eps(f,K) = -\eps_\ell(\pi,K) = -1$ (as $BC_{K_\grl}(\pi_\ell)$ is again the twist
of the special representation by the unique unramified quadratic character and so has root number $+1$). 
If (i) and (iii) hold, then $\eps(f,K) = -\eps_{\ell_1}(\pi,K)\eps_{\ell_2}(\pi,K) = -1$ (as $BC_{K_{\ell_i}}(\pi_{\ell_i})$
is again the special representation and so has root number $-1$).
In particular, for a $K$ satisfying (i), (ii), and (iii) we always have $\eps(f)\eps(f\otimes\chi_D) = \eps(f,K) = -1$,
so $\eps(f\otimes\chi_D) = -\eps(f) = +1$. It then follows from 
\cite[Thm.~B]{Friedberg-Hoffstein} that $K$ can also be chosen to satisfy (iv).
Conditions (i), (ii) and (iii) imply that hypotheses (b), (c), and (d) of Theorem \ref{thmB} hold.
By the work of Gross, Zagier, and Kolyvagin cited in the introduction (alternatively, one could
appeal to results of Kato \cite{Kato}) condition (iv) implies that both $A_f^D(\Q)$ and
$\Sha(A_f^D)$ are finite, from which it then follows that $\rank_\Z A_f(K)=[M_f:\Q]$ and
$\Sha(A_f/K)[p^\infty]$ is finite. By Lemma \ref{rank1lemma}, hypothesis (e)
of Theorem \ref{thmB} then
also holds.  As hypothesis (a) holds by the choice of $p$ and $\lambda$, we conclude from Theorem \ref{thmB} that $\ord_{s=1}L(f,K,s)=1$. As
$L(f,K,s)=L(f,s)L(f\otimes\chi_D,s)$ and $L(f\otimes\chi_D,1)\neq 0$,
it follows that $\ord_{s=1}L(f,s)=1$.

This completes the proof of Theorem A.

The deduction of Theorem \ref{thmC} from Theorem \ref{thmB} is similar: Hypothesis (c)
of Theorem \ref{thmC} is easily 
seen to imply that $H^1_f(\Q,V)\isoarrow H^1_f(\Qp,V)\cong\Qp$. The representation $E[p]$ must be ramified at some odd prime $\ell$ of bad reduction for $E$ (again by Ribet's level-lowering results). An appropriate imaginary quadratic field $K$ that is either inert or ramified at $\ell$ is then chosen, depending on whether $E$ has split or non-split reduction at $\ell$.

\section{A remark on the $r=0$ case}

The arguments used to deduce
Theorem \ref{thmA} from Theorem \ref{thmB} can be adapted to show that if
$A_f(\Q)$ and $\Sha(A_f)$ are finite then $L(f,1)\neq 0$. This gives
an alternate proof of a special case of the results in \cite{SU-MCGL} and \cite{Xin-thesis}
cited in the introduction.

\begin{customthm}{E}\label{thmE} Suppose $N$ is squarefree.
If there is at least one odd prime $\ell$ such that $\pi_\ell$ is the twist of the special
representation by the unique unramified quadratic character or at least two odd primes
$\ell_1$ and $\ell_2$ such that $\pi_{\ell_1}$ and $\pi_{\ell_2}$ are special,
then
\begin{equation*}
\begin{matrix} \#A_f(\Q), \ \#\Sha(A_f)< \infty \implies \ord_{s=1} L(f,s) = 0.
\end{matrix}
\end{equation*}
\end{customthm}

The argument is virtually identical to the proof that Theorem \ref{thmB} implies Theorem \ref{thmA}. The changes
involved are: the hypotheses now imply that $\eps(f)=+1$, and $K$ is chosen to satisfy (i), (ii), (iii),
and
\begin{itemize}
 \item[(iv)$'$] $\ord_{s=1}L(f\otimes\chi_D,s)=1$.
\end{itemize}
This is possible as $\eps(f\otimes\chi_D)$ will equal $-1$. We then conclude from Theorem \ref{thmB},
much as before, that $\ord_{s=1}L(f,K,s)=1$, which implies by the choice of $K$ that $L(f,1)\neq 0$.

\end{document}